\documentclass{article}

\usepackage{arxiv}

\usepackage[utf8]{inputenc} 
\usepackage[T1]{fontenc}    
\usepackage{hyperref}       
\usepackage{url}            
\usepackage{booktabs}       
\usepackage{amsfonts}       
\usepackage{nicefrac}       
\usepackage{microtype}      
\usepackage{lipsum}		
\usepackage{graphicx}
\usepackage{mathtools}
\usepackage{amssymb}
\usepackage{amsthm}
\usepackage{amsmath}
\usepackage{doi}
\usepackage{mathtools, cases}

\allowdisplaybreaks

\newcommand{\sgn}{\operatorname{sgn}}

\newtheorem{theorem}{Theorem}

\title{Stable distributions and pseudo-processes related to fractional Airy functions}

\author{ 
	\href{https://orcid.org/0000-0002-6163-044X}{Manfred Marvin Marchione}\\
	Department of Statistical Sciences\\
	Sapienza University of Rome\\
	\texttt{manfredmarvin.marchione@uniroma1.it} \\
	\And
	\href{https://orcid.org/0000-0002-6421-533X}{Enzo Orsingher} \\
	Department of Statistical Sciences\\
	Sapienza University of Rome\\
	\texttt{enzo.orsingher@uniroma1.it}}

\date{\today}

\renewcommand{\headeright}{}
\renewcommand{\undertitle}{}

\begin{document}
\maketitle

\begin{abstract}
In this paper we study pseudo-processes related to odd-order heat-type equations composed with Lévy stable subordinators. The aim of the article is twofold. We first show that the pseudo-density of the subordinated pseudo-process can be represented as an expectation of damped oscillations with generalized gamma distributed parameters. This stochastic representation also arises as the solution to a fractional diffusion equation, involving a higher-order Riesz-Feller operator, which generalizes the odd-order heat-type equation. We then prove that, if the stable subordinator has a suitable exponent, the time-changed pseudo-process becomes a genuine Lévy stable process. This result permits us to obtain a power series representation for the probability density function of an arbitrary asymmetric stable process of exponent $\nu>1$ and skewness parameter $\beta$, with $0<\lvert\beta\lvert<1$. The methods we use in order to carry out our analysis are based on the study of a fractional Airy function which emerges in the investigation of the higher-order Riesz-Feller operator.
\end{abstract}


\section{Introduction}
\noindent While the analysis of higher-order heat-type equations was first tackled in some special cases by Bernstein \cite{bernstein} and Burwell \cite{burwell}, a probabilistic insight on the topic was established some decades later by Krylov \cite{krylov}. In his paper, the author constructed a signed measure on the space $\Omega$ of measurable functions $x=x(t),\;t>0,$ by defining the cylinder sets
$$I_n=\{x:\;a_k\le x(t_k)\le b_k,\;k=1,...,n\},\qquad 0\le t_0<...< t_{n-1}.$$
\noindent A finitely additive signed measure $\mathbb{P}$ on $\Omega$ is constructed by the rule
$$\mathbb{P}(I_n)=\int_{a_1}^{b_1}\hdots \int_{a_{n}}^{b_{n}}\prod_{k=1}^{n}u(x_k-x_{k-1},t_k-t_{k-1})\;dx_1\hdots dx_{n}$$
\noindent where $x_0=0$, $t_0=0$, $u(x,t)$ is the fundamental solution to the higher-order heat equation
$$\frac{\partial u}{\partial t}=c_m\frac{\partial^m u}{\partial x^m},\qquad m\ge2$$
\noindent and $c_m$ is a suitable coefficient. As pointed out by Miyamoto \cite{miyamoto}, the signed measure $\mathbb{P}$ is well defined on the algebra $\mathcal{F}$ consisting of all cylinder sets and it is finitely additive on $\mathcal{F}$. Since $\mathbb{P}$ is signed, Kolmogorov's extension theorem cannot be applied to extend $\mathbb{P}$ to the $\sigma$-algebra generated by $\mathcal{F}$. Moreover, Krylov \cite{krylov} proved that $\mathbb{P}$ has infinite total variation. Therefore $\mathbb{P}$ cannot be extended to $\sigma\left(\mathcal{F}\right)$.\\
\noindent In Krylov's paper only the even-order case $m=2n$ was considered. Hochberg \cite{hochberg} also developed an It\^{o}-type stochastic calculus for the even-order pseudo-processes. Similar ideas for the construction of pseudo-processes were proposed by Ladokhin \cite{ladokhin}, Daletsky and Fomin \cite{daletsky} and Miyamoto \cite{miyamoto} who investigated equations in which two space derivatives appear. Debbi \cite{debbi1, debbi2} proposed a space-fractional extension of the even-order heat-type equation involving the Riesz operator.\\
\noindent The pseudo-processes related to odd-order heat-type equations were first considered by Orsingher \cite{orsingher} for the third order heat equation and were then extended to all possible orders by Lachal \cite{lachal}. In these papers the main concern was to evaluate the distribution of some functionals of pseudo-processes like the sojourn time or the maximum. Nakajima and Sato \cite{nakajima} have obtained the joint distribution of the hitting time and hitting place for the third order pseudo-process. The sample paths of pseudo-process $x(t)$ display a sort of continuity or moderate discontinuity, as stated by Lachal \cite{lachal3} and Nishioka \cite{nishioka1, nishioka2} by studying the distribution of $x(\tau_a)$ with $$\tau_a=\inf\{t\ge0:\;x(t)>a\}.$$
Alternative approaches have been proposed in the literature for the probabilistic construction of pseudo-processes. Bonaccorsi and Mazzucchi \cite{bonaccorsi} and Lachal \cite{lachal2} obtained the solutions to higher order heat-type equations in terms of the scaling limit of suitable random walks. Orsingher and Toaldo \cite{orsinghertoaldo} constructed the pseudo-processes as the limit of compound Poisson processes with steps represented by pseudo-random variables with signed measures.\\
\noindent The starting point of the present research is the paper by Orsingher and D'Ovidio \cite{orsingherdovidio} in which the fundamental solution to the odd-order heat-type equation
\begin{equation}\label{intro1}
\frac{\partial u}{\partial t}(x,t)=(-1)^n\frac{\partial^{2n+1} u}{\partial x^{2n+1}}(x,t),\qquad x\in\mathbb{R},\;t>0,\;n\in\mathbb{N}\end{equation}
\noindent is expressed as
\begin{equation}\label{intro2}u_{2n+1}(x,t)=\frac{1}{\pi x}\mathbb{E}\left[e^{-b_n x\;G_{2n+1}(1/t)}\sin\left(a_n x\;G_{2n+1}(1/t)\right)\right]\end{equation}
where $G_{\gamma}(\tau)$ is a random variable with generalized gamma distribution having probability density function
$$g_{\gamma}(y;\tau)=\gamma\frac{y^{\gamma-1}}{\tau}\exp\left(-\frac{y^{\gamma}}{\tau}\right),\qquad y,\gamma,\tau>0$$
and $$a_n=\cos\frac{\pi}{2(2n+1)},\qquad b_n=\sin\frac{\pi}{2(2n+1)}.$$
\noindent It is known (see Orsingher \cite{orsingher}) that, in the case $n=1$, the solution to equation (\ref{intro1}) admits the following representation in terms of the Airy function of the first kind:
$$u_3(x,t)=\frac{1}{\sqrt[3]{3t}}\text{Ai}\left(\frac{x}{\sqrt[3]{3t}}\right).$$
\noindent For general values of $n$, we start our analysis by exploring the connection between the probabilistic representation (\ref{intro2}) and the higher-order Airy functions
\begin{equation}\label{intro3}\text{Ai}_{2n+1}(x)=\frac{1}{\pi}\int_0^{+\infty}\cos\left(sx+\frac{s^{2n+1}}{2n+1}\right)ds\qquad n\in\mathbb{N}\end{equation}
\noindent which were profoundly investigated by Askari and Ansari \cite{askari}. We then analyze the pseudo-process $X_{2n+1}(t)$ time-changed with an independent stable subordinator $S_{\nu}(t)$ of exponent $\nu,\;0<\nu<1$.  We show that, as a consequence of the subordination, the probabilistic representation (\ref{intro2}) is transformed into
$$\mathbb{P}(X_{2n+1}(S_{\nu}(t))\in dx)/dx=\frac{1}{\pi x}\mathbb{E}\left[e^{-b_n x\;G_{\nu(2n+1)}(1/t)}\sin\left(a_n x\;G_{\nu(2n+1)}(1/t)\right)\right]$$
where $G_{\gamma}(\tau)$, $a_n$ and $b_n$ are defined as in formula (\ref{intro2}).\\
\noindent In the final section of the paper, we examine a fractional extension of equation (\ref{intro1})
\begin{equation}\label{intro4}
\begin{dcases}
\frac{\partial u}{\partial t}(x,t)={}_xD_{\alpha,\theta}u(x,t),\qquad x\in\mathbb{R},\;t>0,\;\alpha>1\\
u(x,0)=\delta(x).
\end{dcases}
\end{equation}
\noindent where the space operator $D_{\alpha,\theta}$ is a higher-order Riesz-Feller derivative of which we give an explicit representation. The analysis of equation (\ref{intro4}) is carried out by defining a fractional generalization of the Airy function 
$$\normalfont{\text{Ai}}_{\alpha}(x)=\frac{1}{\pi}\int_0^{+\infty}\cos\left(x\gamma+\frac{\gamma^{\alpha}}{\alpha}\right)d\gamma,\qquad \alpha>1$$
\noindent for which we also give the following power series representation:
$$\normalfont{\text{Ai}}_{\alpha}(x)=\frac{1}{\pi\alpha^{\frac{\alpha-1}{\alpha}}}\sum_{k=0}^{\infty}\frac{x^k\alpha^{\frac{k}{\alpha}}}{k!}\Gamma\left(\frac{k+1}{\alpha}\right)\sin\left(\pi\frac{(k+1)(\alpha+1)}{2\alpha}\right).$$
For $\theta=1$ and $\alpha>1$, the solution to the Cauchy problem (\ref{intro4}) is given by
\begin{equation}\label{intro5}u_{\alpha}(x,t)=\frac{1}{(\alpha t)^{\frac{1}{\alpha}}}\normalfont{\text{Ai}}_{\alpha}\left(\frac{x}{(\alpha t)^{\frac{1}{\alpha}}}\right),\qquad x\in\mathbb{R},\;t>0.\end{equation}
\noindent The extension to the case $0<\theta<1$ can be achieved by time-changing the pseudo-process $X_{\alpha}(t)$, having pseudo-distribution (\ref{intro5}), with a stable subordinator $S_{\theta}(t)$ of exponent $0<\theta<1$.\\
\noindent We conclude our analysis by proving that the subordinated pseudo-process $$Y_{\alpha,\theta}(t)=X_{\alpha}(S_{\theta}(t))$$
\noindent is, for $0<\alpha\theta<2$, a stable process with characteristic function $$\mathbb{E}\left[e^{i\gamma Y_{\alpha,\theta}(t)}\right]=e^{-t\lvert \gamma\lvert^{\alpha\theta}\cos\frac{\pi\theta}{2}\left(1+i\tan\frac{\pi\theta}{2}\sgn\gamma\right)}$$ and its pseudo-density becomes a genuine probability density function. This implies that, by giving a power series representation for the distribution of $Y_{\alpha,\theta}(t)$, we are able to obtain the exact distribution of an arbitrary asymmetric stable process of exponent $1<\nu<2$ and skewness parameter $0<\beta<1$. Our result is consistent with that reported in the book by Zolotarev \cite{zolotarev} and shows that stable processes can be represented as pseudo-processes time-changed with stable subordinators.


\section{Odd-order pseudo-processes and Airy functions}
\noindent Explicit representations for the solution to the odd-order heat-type equation
\begin{equation}\label{oddorder}
\begin{dcases}
\frac{\partial u}{\partial t}(x,t)=(-1)^n\frac{\partial^{2n+1} u}{\partial x^{2n+1}}(x,t),\qquad x\in\mathbb{R},\;t>0,\;n\in\mathbb{N}\\
u(x,0)=\delta(x)
\end{dcases}\end{equation}
\noindent have been proposed in different forms in the literature. The connection between the solution to equation (\ref{oddorder}) and the higher-order Airy function
\begin{equation}\label{generalairydef}\text{Ai}_{2n+1}(x)=\frac{1}{\pi}\int_0^{+\infty}\cos\left(sx+\frac{s^{2n+1}}{2n+1}\right)ds\qquad n\in\mathbb{N}\end{equation}
\noindent was highlighted by Askari and Ansari \cite{askari}. The authors studied the generalized Airy function (\ref{generalairydef}) as a solution to the ordinary differential equation
\begin{equation}\label{a2n}y^{(2n)}(x)+(-1)^nxy(x)=0.\end{equation}
\noindent Similarly to the classical Airy equation, equation (\ref{a2n}) can be solved by applying the generalized Laplace transform method. By splitting then the complex plane into $2(2n+1)$ sectors of amplitude $\frac{\pi}{2(2n+1)}$, $2n$ linearly independent solutions are obtained among which the solution (\ref{generalairydef}) arises.\\
\noindent In the same paper the authors pointed out that, by solving equation (\ref{oddorder}) with a Fourier transform approach, the solution reads
\begin{align}u_{2n+1}(x,t)=&\frac{1}{2\pi}\int_{-\infty}^{+\infty}e^{-i\gamma x - it\gamma^{2n+1}}d\gamma\nonumber\\
=&\frac{1}{\pi}\int_{0}^{+\infty}\cos\left(\gamma x+t\gamma^{2n+1}\right)d\gamma\nonumber\\
=&\frac{1}{t^{\frac{1}{2n+1}}\;(2n+1)^{\frac{1}{2n+1}}}\;\text{Ai}_{2n+1}\left(\frac{x}{t^{\frac{1}{2n+1}}\;(2n+1)^{\frac{1}{2n+1}}}\right).\label{generalairy}
\end{align}
The representation (\ref{generalairy}) is crucial for our purposes since most of our probabilistic considerations emerge from the study of generalized Airy functions.\\
\noindent For $n=1$, the classical Airy function admits the well-known power series representation
\begin{equation}\label{classicalairypowerseries}\text{Ai}_3(z)=\frac{2}{3^{\frac{7}{6}}}\sum_{k=0}^{+\infty}\left(\frac{z}{3^{\frac{2}{3}}}\right)^k\frac{\sin\left(\frac{2\pi}{3}(k+1)\right)}{\Gamma\left(\frac{k}{3}+1\right)\Gamma\left(\frac{k+2}{3}\right)}.\end{equation}
The power series expansion (\ref{classicalairypowerseries}) can be extended to the higher-order Airy function (\ref{generalairydef}) as shown by Ansari and Askari \cite{ansari}. In the following theorem we provide a different proof which generalizes the one proposed by Watson \cite{watson}, pag. 189, for the classical case $n=1$.
\smallskip

\begin{theorem}\label{airythm}The higher-order Airy function (\ref{generalairydef}) admits the following series representation for $x\in\mathbb{R}$:
\begin{equation}\label{airyseries}\normalfont{\text{Ai}}_{2n+1}(x)=\frac{1}{\pi(2n+1)^{\frac{2n}{2n+1}}}\sum_{k=0}^{\infty}\frac{x^k(2n+1)^{\frac{k}{2n+1}}}{k!}\sin\left(\pi\frac{(k+1)(n+1)}{2n+1}\right)\Gamma\left(\frac{k+1}{2n+1}\right)\end{equation}\end{theorem}
\begin{proof}We start by observing that \begin{equation}\label{integrand}\text{Ai}_{2n+1}(x)=\frac{1}{2\pi}\int_{-\infty}^{+\infty}\exp\left(itx+i\frac{t^{2n+1}}{2n+1}\right)dt.\end{equation}
Consider a contour enclosing two circular sectors of radius $R$, centered at the origin, such that the arc of the first sector starts at the point $R$ and ends at $R\exp\{i\pi\frac{1}{4n+2}\}$ while the second sector has an arc starting at
$R\exp\{i\pi\frac{4n+1}{4n+2}\}$ and ending at $-R$ (see figure \ref{fig:contour1}). By the Cauchy integral theorem, the integrand of formula (\ref{integrand}) has null integral along the considered contour. Moreover, because of Jordan's lemma, the integral along the arcs of the circular sectors converge to 0 as $R\to+\infty$.\\

\begin{figure}[t]
\centering
\includegraphics[scale=0.9]{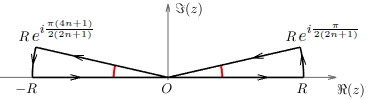}

\caption{the contour used for the proof of theorem \ref{airythm} encloses two circular sectors of radius $R$ centered at the origin. The arc of the first sector starts at the point $R$ and ends at $R\exp\{i\pi\frac{1}{4n+2}\}$ while the second sector has an arc starting at $R\exp\{i\pi\frac{4n+1}{4n+2}\}$ and ending at $-R$. The integral along the arcs converges to 0 as $R\to+\infty$.}
\label{fig:contour1}
\end{figure}

\noindent Thus, by summing the integrals along the two contours and taking the limit for $R\to+\infty$, we can write
\begin{align}\text{Ai}_{2n+1}&(x)=\frac{1}{2\pi}e^{i\pi\frac{1}{4n+2}}\int_0^{+\infty}\exp\left\{rxe^{i\pi\frac{n+1}{2n+1}}-\frac{r^{2n+1}}{2n+1}\right\}dr\nonumber\\
&\hspace{8mm}-\frac{1}{2\pi}e^{i\pi\frac{4n+1}{4n+2}}\int_0^{+\infty}\exp\left\{rxe^{i\pi\frac{3n+1}{2n+1}}-\frac{r^{2n+1}}{2n+1}\right\}dr\nonumber\\
=&\frac{1}{2\pi}e^{i\pi\frac{1}{4n+2}}\int_0^{+\infty}\exp\left\{rxe^{i\pi\frac{n+1}{2n+1}}-\frac{r^{2n+1}}{2n+1}\right\}dr\nonumber\\
\;&+\frac{1}{2\pi}e^{-i\pi\frac{1}{4n+2}}\int_0^{+\infty}\exp\left\{rxe^{-i\pi\frac{n+1}{2n+1}}-\frac{r^{2n+1}}{2n+1}\right\}dr.\label{domconv}\end{align}
\noindent The exponential functions $f(r)=\exp\left\{rxe^{\pm i\pi\frac{n+1}{2n+1}}\right\}$ in formula (\ref{domconv}) can be substituted by their Taylor expansions and the order of integration and summation can then be interchanged in force of the dominated convergence theorem. In order to apply the dominated convergence theorem, observe that $f(r)=\lim_{n\to+\infty}f_n(r)$ where $f_n(r)$ represents the sum of the first $n$ terms of the Taylor expansion of $f(r)$.  Since $\lvert f_n(r)\lvert \le e^{\lvert rx\lvert }$ for all $n$ and $e^{\lvert rx\lvert }$ is integrable on $(0,+\infty)$ with respect to the measure $\mu(dr)=e^{-\frac{r^{2n+1}}{2n+1}}dr$, the dominated convergence theorem can be applied. Thus, we can write
\begin{align}\text{Ai}_{2n+1}(x)=&\frac{1}{2\pi}e^{i\pi\frac{1}{4n+2}}\sum_{k=0}^{\infty}\frac{x^ke^{ik\pi\frac{n+1}{2n+1}}}{k!}\int_0^{+\infty}r^k\exp\left\{-\frac{r^{2n+1}}{2n+1}\right\}dr\nonumber\\
\;&+\frac{1}{2\pi}e^{-i\pi\frac{1}{4n+2}}\sum_{k=0}^{\infty}\frac{x^ke^{-ik\pi\frac{n+1}{2n+1}}}{k!}\int_0^{+\infty}r^k\exp\left\{-\frac{r^{2n+1}}{2n+1}\right\}dr\nonumber\\
=&\frac{1}{\pi(2n+1)^{\frac{2n}{2n+1}}}\sum_{k=0}^{\infty}\frac{e^{i\pi\frac{2k(n+1)+1}{4n+2}}+e^{-ik\pi\frac{2k(n+1)+1}{4n+2}}}{2}\;x^k(2n+1)^{\frac{k}{2n+1}}\frac{\Gamma\left(\frac{k+1}{2n+1}\right)}{k!}\nonumber\\
=&\frac{1}{\pi(2n+1)^{\frac{2n}{2n+1}}}\sum_{k=0}^{\infty}\frac{x^k(2n+1)^{\frac{k}{2n+1}}}{k!}\cos\left(\pi\frac{2k(n+1)+1}{4n+2}\right)\Gamma\left(\frac{k+1}{2n+1}\right)\nonumber\\
=&\frac{1}{\pi(2n+1)^{\frac{2n}{2n+1}}}\sum_{k=0}^{\infty}\frac{x^k(2n+1)^{\frac{k}{2n+1}}}{k!}\sin\left(\pi\frac{(k+1)(n+1)}{2n+1}\right)\Gamma\left(\frac{k+1}{2n+1}\right)\nonumber\end{align}
where in the last step we have used the formula $\cos \theta = \sin\left(\theta+\frac{\pi}{2}\right)$.
\end{proof}
\noindent We observe that formula (\ref{airyseries}) is consistent with the result obtained by Orsingher and D'Ovidio \cite{orsingherdovidio} who proved the following representation for the pseudo-density:
\begin{equation}\label{uodd}u_{2n+1}(x,t)=-\frac{1}{\pi x}\sum_{k=1}^{\infty}\frac{1}{k!}\sin\left(\frac{n\pi k}{2n+1}\right)\Gamma\left(1+\frac{k}{2n+1}\right)\left(-\frac{x}{t^{\frac{1}{2n+1}}}\right)^k.\end{equation}
\noindent The expression (\ref{uodd}) is immediately obtained by combining formulas (\ref{generalairy}) and (\ref{airyseries}). Moreover, by using the triplication formula for the gamma function
$$\Gamma(3z)=\frac{1}{2\pi}3^{3z-\frac{1}{2}}\Gamma\left(z\right)\Gamma\left(z+\frac{1}{3}\right)\Gamma\left(z+\frac{2}{3}\right)$$
with $z=\frac{k+1}{3}$, formula (\ref{airyseries}) reduces to (\ref{classicalairypowerseries}) for $n=1$.\\
Orsingher and D'Ovidio \cite{orsingherdovidio} discussed the behaviour of the function $u_{2n+1}(x,t)$ and they observed that, while for $n=1$ the pseudo-density is non-negative for $x>0$, the non-negativity on the positive semi-axis is lost for $n>1$ due to the oscillating behaviour of the function. They also pointed out that the asymmetry of $u_{2n+1}(x,t)$ seems to reduce as $n$ increases. This assertion can be supported by observing that
\begin{align}\lim_{n\to+\infty}u_{2n+1}(x,t)=&-\frac{1}{\pi x}\sum_{k=1}^{+\infty}\frac{(-x)^k}{k!}\sin\left(k\frac{\pi}{2}\right)\nonumber\\
=&-\frac{1}{2\pi ix}\sum_{k=0}^{+\infty}\frac{(-x)^k}{k!}\left(e^{ik\frac{\pi}{2}}-e^{-ik\frac{\pi}{2}}\right)\nonumber\\
=&-\frac{1}{2\pi ix}\left(e^{-ix}-e^{ix}\right)=\frac{\sin x}{\pi x}.\nonumber
\end{align}


\section{Pseudo-processes time-changed with stable subordinators}
\noindent  In this section we study the pseudo-process $X_{2n+1}(t)$, governed by the higher-order equation (\ref{oddorder}), time-changed with stable subordinators. In particular, we consider stable subordinators $S_{\theta}(t)$ having characteristic function
$$\mathbb{E}\left[e^{i\gamma S_{\theta}(t)}\right]=e^{-t\lvert \gamma\lvert^{\theta}e^{-i\frac{\theta\pi}{2}\sgn{\gamma}}}.$$
\noindent In the proof of the following theorem we need, as a preliminary result, the Mellin transform of the Wright function (see Prudnikov et al. \cite{prudnikov}, pag. 355)
\begin{equation*}\int_0^{+\infty}\text{W}_{a,b}(-x)\;x^{\eta-1}dx=\frac{\Gamma(\eta)}{\Gamma(b-a\eta)},\qquad\eta>0\end{equation*}
\noindent where 
\begin{equation*}\text{W}_{a,b}(z)=\sum_{k=0}^{\infty}\frac{z^k}{k!\;\Gamma(ak+b)},\qquad z\in\mathbb{C},\;a>-1,\;b\in\mathbb{C}.\end{equation*}

\begin{theorem} \label{gennu}Let $X_{2n+1}(t)$ be the pseudo-process governed by equation (\ref{oddorder}) and $S_{\theta}(t)$ a stable subordinator of exponent $\theta$, $0<\theta<1$, independent of $X_{2n+1}(t)$. For $\theta(2n+1)>1$ the following formula holds:
\begin{equation}\label{probext}\mathbb{P}(X_{2n+1}(S_{\theta}(t))\in dx)/dx=\frac{1}{\pi x}\mathbb{E}\left[e^{-b_n x\;G_{\theta(2n+1)}(1/t)}\sin\left(a_n x\;G_{\theta(2n+1)}(1/t)\right)\right]\end{equation}\end{theorem}
\noindent where $G_{\gamma}(\tau)$ is a random variable with generalized gamma distribution having probability density function
$$g_{\gamma}(y;\tau)=\gamma\frac{y^{\gamma-1}}{\tau}\exp\left(-\frac{y^{\gamma}}{\tau}\right),\qquad y,\gamma,\tau>0$$
and $$a_n=\cos\frac{\pi}{2(2n+1)},\qquad b_n=\sin\frac{\pi}{2(2n+1)}.$$
\begin{proof}
We use the Wright function representation of the probability density function of the stable subordinator
\begin{align}h_{\theta}(x,t)=&\frac{\theta t}{x^{\theta+1}}W_{-\theta,1-\theta}\left(-\frac{t}{x^{\theta}}\right)\nonumber\\
=&\frac{\theta}{x}\sum_{k=0}^{\infty}\frac{(-1)^kt^{k+1}}{x^{\theta(k+1)}k!\;\Gamma(-\theta(k+1)+1)}.\end{align}
Thus we have
\begin{align}\mathbb{P}&(X_{2n+1}(S_{\theta}(t))\in dx)/dx=\int_{0}^{+\infty}u_{2n+1}(x,s)\;h_{\theta}(s,t)\;ds\nonumber\\
=&-\frac{\theta t}{\pi x}\sum_{k=1}^{\infty}\frac{\left(-x\right)^k}{k!}\sin\left(\frac{n\pi k}{2n+1}\right)\Gamma\left(1+\frac{k}{2n+1}\right)\int_0^{+\infty}s^{-\frac{k}{2n+1}-\theta-1}\text{W}_{-\theta,1-\theta}\left(-\frac{t}{s^\theta}\right)ds\nonumber\end{align}
\noindent Observe that, in the last step, we have used the power series representation (\ref{uodd}) for $u_{2n+1}(x,s)$ and we have interchanged the order of summation and integration by the dominated convergence theorem. We note that $u_{2n+1}(x,s)=\lim_{m\to+\infty}f_m(x,s)$ where $f_m(x,s)$ represents the sum of the first $m$ terms of the power series (\ref{uodd}). Of course $\lvert f_m(x,s)\lvert \le g(x,s)$ for all $m$, where
$g(x,s)=\frac{1}{\pi x}\sum_{k=1}^{\infty}\frac{1}{k!}\Gamma\left(1+\frac{k}{2n+1}\right)\left(\frac{\lvert x\lvert }{s^{\frac{1}{2n+1}}}\right)^k$. By applying the monotone convergence theorem and integrating termwise, it can be shown that $\int_{0}^{+\infty}g(x,s)\;h_{\theta}(s,t)\;ds<+\infty$ for $\theta(2n+1)>1$. Thus, the dominated convergence theorem can be applied and we can write
\begin{align}\mathbb{P}&(X_{2n+1}(S_{\theta}(t))\in dx)/dx\nonumber\\
=&-\frac{1}{\pi x}\sum_{k=1}^{\infty}\frac{1}{k!}\sin\left(\frac{n\pi k}{2n+1}\right)\Gamma\left(1+\frac{k}{\theta(2n+1)}\right)\left(-\frac{x}{t^{\frac{1}{\theta(2n+1)}}}\right)^k\nonumber\\
=&-\frac{(2n+1)t}{\pi x}\int_0^{+\infty}s^{2n}e^{-ts^{2n+1}}\sum_{k=1}^{\infty}\frac{1}{k!}\sin\left(\frac{n\pi k}{2n+1}\right)\left(-xs^{\frac{1}{\theta}}\right)^k\;ds.\nonumber
\end{align}
\noindent By using the relationship 
\begin{equation}\label{sineseries}e^{x\cos\phi}\sin(x\sin\phi)=\sum_{k=0}^{\infty}\sin(\phi k)\frac{x^k}{k!}\end{equation}
\noindent we finally obtain
\begin{align}\mathbb{P}(X_{2n+1}(&S_{\theta}(t))\in dx)/dx\nonumber\\
=&\frac{(2n+1)t}{\pi x}\int_0^{+\infty}e^{-xs^{\frac{1}{\theta}}\cos\frac{n\pi}{2n+1}}\sin\left(xs^{\frac{1}{\theta}}\sin\frac{n\pi}{2n+1}\right)s^{2n}e^{-ts^{2n+1}}ds\nonumber\\
=&\frac{(2n+1)t}{\pi x}\int_0^{+\infty}e^{-xs^{\frac{1}{\nu}}\sin\frac{\pi}{2(2n+1)}}\sin\left(xs^{\frac{1}{\theta}}\cos\frac{\pi}{2(2n+1)}\right)s^{2n}e^{-ts^{2n+1}}ds.\nonumber\end{align}
\noindent The change of variables $y=s^{\frac{1}{\theta}}$ completes the proof.
\end{proof}

\noindent Formula (\ref{probext}) provides a probabilistic representation for the pseudo-density of the subordinated pseudo-process $X_{2n+1}(S_{\theta}(t))$ in terms of an expected value of damped oscillations with generalized gamma distributed parameters. Theorem \ref{gennu} is an extension of the probabilistic representation (\ref{intro2}) for the odd-order Airy function obtained by Orsingher and D'Ovidio \cite{orsingherdovidio}. As a corollary of our result we can write
\begin{equation}\label{coroll}\mathbb{P}(X_{2n+1}(S_{\theta}(t))\in dx)/dx=-\frac{1}{\pi x}\sum_{k=1}^{\infty}\frac{1}{k!}\sin\left(\frac{n\pi k}{2n+1}\right)\Gamma\left(1+\frac{k}{\theta(2n+1)}\right)\left(-\frac{x}{t^{\frac{1}{\theta(2n+1)}}}\right)^k.\end{equation}
The condition $\theta(2n+1)>1$ imposed in theorem \ref{gennu} ensures that the power series (\ref{coroll}) has infinite radius of convergence. This can be proved by using the Stirling approximation formula for the gamma function.\\
\noindent In the next section we extend formula (\ref{coroll}) to non-integer values of $n$ by introducing a suitable fractional generalization of the pseudo-process $X_{2n+1}(t)$. Moreover, we show that the pseudo-density (\ref{coroll}) and its fractional extension are genuine non-negative probability density functions if the parameters are chosen in a suitable way.

\section{Fractional Airy functions and stable processes}
\noindent In this section we generalize the results obtained so far by studying a family of fractional-order pseudo-processes. In particular, we are interested in the solution to the fractional partial differential equation
\begin{equation}\label{RFeq}
\begin{dcases}
\frac{\partial u}{\partial t}(x,t)={}_{x}D^{\alpha}_{\theta}u(x,t),\qquad \alpha>1,\;\theta\in(0,1]\\
u(x,0)=\delta(x)
\end{dcases}\end{equation}
where ${}_{x}D^{\alpha}_{\theta}$ represents the Riesz-Feller fractional derivative (originally defined by Feller \cite{feller}) for which we modify the usual restrictions $0<\alpha\le2$ and $\lvert \theta\lvert \le\min\{\alpha,2-\alpha\}$. As pointed out by Mainardi \cite{mainardi}, these restrictions are usually imposed in order to ensure the probabilistic interpretability of the Riesz-Feller operator. For the restricted parameters, the solution to equation (\ref{RFeq}) is indeed the probability density function of an asymmetric stable process. However, by eliminating the upper bound on $\alpha$, we are able to obtain an explicit series representation for the density function of asymmetric stable processes of exponent $\nu>1$ and skewness parameter $\beta$, with $0<\lvert \beta\lvert <1$.\\
\noindent By using the notation $$\mathcal{F}\{f(x)\}(\gamma)=\int_{-\infty}^{+\infty}e^{i\gamma x}f(x)dx$$
\noindent the Riesz-Feller fractional operator can be defined implicitely by means of its Fourier transform
\begin{equation}\label{RFfouriertr}\mathcal{F}\{{}_{x}D^{\alpha}_{\theta}f(x)\}(\gamma)=-\lvert \gamma\lvert^{\alpha}e^{\frac{i\pi\theta}{2}\sgn(\gamma)}\mathcal{F}\{f(x)\}(\gamma).\end{equation}

\noindent For functions $f:\mathbb{R}\to\mathbb{R}$, $f\in\mathcal{C}^m(\mathbb{R})$, with derivatives decaying for $\lvert x\lvert \to+\infty$, the Riesz-Feller derivative of order $\alpha$ admits, for all values of $\theta$, the explicit integral representation
\begin{align}{}_{x}D^{\alpha}_{\theta}f(x)=\frac{\Gamma(\alpha-m+1)}{\pi}\frac{d^m}{dx^m}\left[\sin\frac{\pi(\alpha+\theta)}{2}\int_0^{+\infty}\frac{f(x+z)}{z^{\alpha-m+1}}dz\right.\nonumber\\
\left.+(-1)^m\sin\frac{\pi(\alpha-\theta)}{2}\int_0^{+\infty}\frac{f(x-z)}{z^{\alpha-m+1}}dz\right]\label{explicitRiesz}\end{align}
\noindent with $m-1<\alpha<m$, $m\in\mathbb{N}$.\\
\noindent The integral representation (\ref{explicitRiesz}) can be proved by checking that its Fourier transform coincides with the expression (\ref{RFfouriertr}) (see Orsingher and Toaldo \cite{orsinghertoaldo}).\\

\noindent We start our analysis by studying equation (\ref{RFeq}) for $\theta=1$:
\begin{equation}\label{RFtheta1}
\begin{dcases}
\frac{\partial u}{\partial t}(x,t)={}_{x}D^{\alpha}_{1}u(x,t),\qquad \alpha>1\\
u(x,0)=\delta(x).
\end{dcases}\end{equation}
Denoting by $u_{\alpha}(x,t)$ the solution to equation (\ref{RFtheta1}), its Fourier transform with respect to $x$ is
\begin{equation}\label{fractionalcharfun}\mathcal{F}\{u_{\alpha}(x,t)\}(\gamma,t)=e^{-it\sgn(\gamma)\lvert \gamma\lvert^{\alpha}}\end{equation}
from which we obtain
\begin{align}u_{\alpha}(x,t)=&\frac{1}{2\pi}\int_{-\infty}^{+\infty}e^{-i\gamma x-it\sgn(\gamma)\lvert \gamma\lvert^{\alpha}}d\gamma\nonumber\\
=&\frac{1}{\pi}\int_0^{+\infty}\cos\left(\gamma x+t\lvert \gamma\lvert^{\alpha}\right)d\gamma.\label{ualpha}\end{align}
In analogy with formula (\ref{generalairy}), by defining the generalized Airy function
\begin{equation}\label{definefractionalairy}\normalfont{\text{Ai}}_{\alpha}(x)=\frac{1}{\pi}\int_0^{+\infty}\cos\left(xs+\frac{s^{\alpha}}{\alpha}\right)ds,\qquad \alpha>1\end{equation}
we express the pseudo-density (\ref{ualpha}) in the form
\begin{equation}\label{fractionaldensity}u_{\alpha}(x,t)=\frac{1}{(\alpha t)^{\frac{1}{\alpha}}}\normalfont{\text{Ai}}_{\alpha}\left(\frac{x}{(\alpha t)^{\frac{1}{\alpha}}}\right).\end{equation}
The first problem we tackle is the convergence of the integral defining the generalized Airy function (\ref{definefractionalairy}).

\begin{theorem}\label{improperthm}For $\alpha>1$, the improper integral defining the generalized Airy function (\ref{definefractionalairy}) is convergent $\forall x\in\mathbb{R}$.\end{theorem}
\begin{proof}For any $x\in\mathbb{R}$ there exists $s_0$ such that the function $$u(s)=xs+\frac{s^{\alpha}}{\alpha}$$
\noindent  is increasing for $s\ge s_0$. Denote by $s(u)$ the inverse of $u(s)$ for $s>s_0$. For $c>s_0$ we have that
\begin{align}\int_{s_0}^c \cos(u(s))ds=\int_{u(s_0)}^{u(c)}\frac{\cos(u)}{x+s(u)^{\alpha-1}}du.\label{airyconv}\end{align}
\noindent Since $\lim_{u\to+\infty}s(u)=+\infty$ and $\alpha>1$, for any $x\in\mathbb{R}$ the function $u\mapsto \frac{1}{x+s(u)^{\alpha-1}}$ converges to 0 for $u\to+\infty$. Thus, by taking the limit for $c\to+\infty$ of the integrals in formula (\ref{airyconv}), the Dirichlet test implies the convergence of the integral
$$\int_{s_0}^{+\infty}\cos(u(s))ds.$$\noindent 
The convergence of the integral in formula (\ref{definefractionalairy}) follows immediately.\end{proof}

\noindent In the following theorem we obtain a power series representation for the generalized Airy function.
\begin{theorem}\label{gafthm}The generalized Airy function $$\normalfont{\text{Ai}}_{\alpha}(x)=\frac{1}{\pi}\int_0^{+\infty}\cos\left(xs+\frac{s^\alpha}{\alpha}\right)ds,\qquad\alpha>1$$ admits the power series representation
\begin{equation}\normalfont{\text{Ai}}_{\alpha}(x)=\frac{1}{\pi\alpha^{\frac{\alpha-1}{\alpha}}}\sum_{k=0}^{\infty}\frac{x^k\alpha^{\frac{k}{\alpha}}}{k!}\Gamma\left(\frac{k+1}{\alpha}\right)\sin\left(\pi\frac{(k+1)(\alpha+1)}{2\alpha}\right).\end{equation}\end{theorem}
\begin{proof}
We start by observing that
\begin{equation}\normalfont{\text{Ai}}_{\alpha}(x)=
\frac{1}{2\pi}\left(\int_{0}^{+\infty}\exp\left\{itx+i\frac{t^{\alpha}}{\alpha}\right\}dt+\int_{0}^{+\infty}\exp\left\{-itx-i\frac{t^{\alpha}}{\alpha}\right\}dt\right).\label{SgnStep}\end{equation}
\noindent By setting $$f_1(t)=\exp\left\{itx+i\frac{t^{\alpha}}{\alpha}\right\},\qquad f_2(t)=\exp\left\{-itx-i\frac{t^{\alpha}}{\alpha}\right\}$$
\noindent the expression (\ref{SgnStep}) can be reformulated by integrating $f_1$ and $f_2$ along two suitable contours encircling two different sectors of a circular annulus with inner radius $\varepsilon$ and outer radius $R$. The contours are described in detail in figure \ref{fig:contour2}.\\

\begin{figure}[b]
\centering
\includegraphics[scale=0.9]{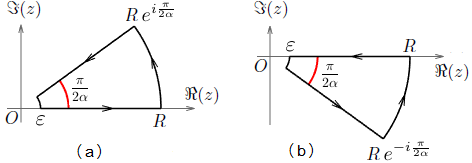}

\caption{the functions $f_1$ and $f_2$ are integrated respectively along the contours $(a)$ and $(b)$. Both contours enclose sectors of a circular annulus with inner radius $\varepsilon$ and outer radius $R$. The contour $(a)$ has arcs with angle ranging from $0$ to $\frac{\pi}{2\alpha}$, while the arcs of the contour $(b)$ have angle ranging from $-\frac{\pi}{2\alpha}$ to $0$. The integrals along the arcs converge to 0 as $\varepsilon\to0$ and $R\to+\infty$.}
\label{fig:contour2}
\end{figure}

\noindent By applying the Cauchy integral theorem and by taking the limits for $\varepsilon\to0$ and $R\to+\infty$, we obtain
\begin{align}\normalfont{\text{Ai}}_{\alpha}(x)&=\frac{1}{2\pi}e^{i\frac{\pi}{2\alpha}}\int_{0}^{+\infty}\exp\left\{ixre^{i\frac{\pi}{2\alpha}}-\frac{r^{\alpha}}{\alpha}\right\}dr\nonumber\\
&\qquad+\frac{1}{2\pi}e^{-i\frac{\pi}{2\alpha}}\int_{0}^{+\infty}\exp\left\{-ixre^{-i\frac{\pi}{2\alpha}}-\frac{r^{\alpha}}{\alpha}\right\}dr\nonumber\\
&=\frac{1}{2\pi}e^{i\frac{\pi}{2\alpha}}\sum_{k=0}^{\infty}\frac{x^ke^{ik\pi\frac{\alpha+1}{2\alpha}}}{k!}\int_{0}^{+\infty}r^ke^{-\frac{r^{\alpha}}{\alpha}}dr\nonumber\\
&\qquad+\frac{1}{2\pi}e^{-i\frac{\pi}{2\alpha}}\sum_{k=0}^{\infty}\frac{x^ke^{-ik\pi\frac{\alpha+1}{2\alpha}}}{k!}\int_{0}^{+\infty}r^ke^{-\frac{r^{\alpha}}{\alpha}}dr\nonumber\\
&=\frac{1}{\pi\alpha^{\frac{\alpha-1}{\alpha}}}\sum_{k=0}^{\infty}\frac{e^{i\pi\frac{k(\alpha+1)+1}{2\alpha}}+e^{-i\pi\frac{k(\alpha+1)+1}{2\alpha}}}{2}\frac{x^k\alpha^{\frac{k}{\alpha}}}{k!}\Gamma\left(\frac{k+1}{\alpha}\right)\nonumber\\
&=\frac{1}{\pi\alpha^{\frac{\alpha-1}{\alpha}}}\sum_{k=0}^{\infty}\frac{x^k\alpha^{\frac{k}{\alpha}}}{k!}\Gamma\left(\frac{k+1}{\alpha}\right)\cos\left(\pi\frac{k(\alpha+1)+1}{2\alpha}\right)\nonumber\\
&=\frac{1}{\pi\alpha^{\frac{\alpha-1}{\alpha}}}\sum_{k=0}^{\infty}\frac{x^k\alpha^{\frac{k}{\alpha}}}{k!}\Gamma\left(\frac{k+1}{\alpha}\right)\sin\left(\pi\frac{(k+1)(\alpha+1)}{2\alpha}\right)\nonumber\end{align}
where in the last step we have used the formula $\cos \theta = \sin\left(\theta+\frac{\pi}{2}\right)$. The inversion of the order of summation and integration performed in the proof can be justified again as in theorem \ref{airythm}.
\end{proof}

\noindent By using formula (\ref{fractionaldensity}) and by applying theorem \ref{gafthm} we can now express the pseudo-density of the fractional order pseudo-process in the form
$$u_{\alpha}(x,t)=-\frac{1}{\pi x}\sum_{k=1}^{\infty}\left(-\frac{x}{t^{\frac{1}{\alpha}}}\right)^k\frac{\Gamma\left(1+\frac{k}{\alpha}\right)}{k!}\sin\left(k\pi\frac{\alpha-1}{2\alpha}\right).$$
\noindent Our last step is to time-change the fractional order pseudo-process by means of stable subordinators. Similarly to the the odd-order case, the following result holds.

\begin{theorem}\label{stablethm}Let $X_{\alpha}(t)$ be the pseudo-process governed by the fractional equation (\ref{RFtheta1}) and $S_{\theta}(t)$ be a stable subordinator of exponent $\theta$, $0<\theta<1$, independent of $X_{\alpha}(t)$. For $\alpha\theta>1$ the following formula holds:
\begin{equation}\mathbb{P}(X_{\alpha}(S_{\theta}(t))\in dx)/dx=\frac{1}{\pi x}\mathbb{E}\left[e^{-b_{\alpha} x\;G_{\alpha\theta}(1/t)}\sin\left(a_{\alpha} x\;G_{\alpha\theta}(1/t)\right)\right]\end{equation}\end{theorem}
\noindent where $G_{\gamma}(\tau)$ is a random variable with generalized gamma distribution having probability density function
$$g_{\gamma}(y;\tau)=\gamma\frac{y^{\gamma-1}}{\tau}\exp\left(-\frac{y^{\gamma}}{\tau}\right),\qquad y,\gamma,\tau>0$$
and $$a_{\alpha}=\sin\frac{\pi}{2\alpha},\qquad b_{\alpha}=\cos\frac{\pi}{2\alpha}.$$
\begin{proof}
As in theorem \ref{gennu}, we have that
\begin{align}\mathbb{P}&(X_{\alpha}(S_{\theta}(t))\in dx)/dx=\int_{0}^{+\infty}u_{\alpha}(x,s)\;h_{\theta}(s,t)\;ds\nonumber\\
=&-\frac{\theta t}{\pi x}\sum_{k=1}^{\infty}\frac{\left(-x\right)^k}{k!}\sin\left(k\pi\frac{\alpha-1}{2\alpha}\right)\Gamma\left(1+\frac{k}{\alpha}\right)\int_0^{+\infty}s^{-\frac{k}{\alpha}-\theta-1}\text{W}_{-\theta,1-\theta}\left(-\frac{t}{s^\theta}\right)ds\nonumber\\
=&-\frac{1}{\pi x}\sum_{k=1}^{\infty}\frac{1}{k!}\sin\left(k\pi\frac{\alpha-1}{2\alpha}\right)\Gamma\left(1+\frac{k}{\theta\alpha}\right)\left(-\frac{x}{t^{\frac{1}{\theta\alpha}}}\right)^k.\nonumber
\end{align}
\noindent The obtained power series can be shown to be convergent for $\alpha\theta>1$ by using the Stirling approximation formula for the gamma function. By using now the integral representation of the gamma function and formula (\ref{sineseries}) as in theorem \ref{gennu}, we finally obtain
\begin{align}\mathbb{P}(X_{\alpha}(&S_{\theta}(t))\in dx)/dx\nonumber\\
=&-\frac{\alpha t}{\pi x}\int_0^{+\infty}s^{\alpha-1}e^{-ts^{\alpha}}\sum_{k=1}^{\infty}\frac{1}{k!}\sin\left(k\pi\frac{\alpha-1}{2\alpha}\right)\left(-xs^{\frac{1}{\theta}}\right)^k\;ds\nonumber\\
=&\frac{\alpha t}{\pi x}\int_0^{+\infty}e^{-xs^{\frac{1}{\theta}}\cos\frac{\pi(\alpha-1)}{2\alpha}}\sin\left(xs^{\frac{1}{\theta}}\sin\frac{\pi(\alpha-1)}{2\alpha}\right)s^{\alpha-1}e^{-ts^{\alpha}}ds\nonumber\\
=&\frac{\alpha t}{\pi x}\int_0^{+\infty}e^{-xs^{\frac{1}{\theta}}\sin\frac{\pi}{2\alpha}}\sin\left(xs^{\frac{1}{\theta}}\cos\frac{\pi}{2\alpha}\right)s^{\alpha-1}e^{-ts^{\alpha}}ds.\nonumber\end{align}
The change of variables $y=s^{\frac{1}{\theta}}$ completes the proof.
\end{proof}

\noindent We conclude our analysis by showing that theorem \ref{stablethm} permits us to obtain a series representation for the exact distribution of asymmetric stable processes with exponent $1<\nu<2$, $\nu$ being related to $\alpha$ and $\theta$, and skewness parameter $\beta$, $0<\lvert \beta\lvert <1$.\\
\noindent Consider the pseudo-process $X_{\alpha}(t)$ and an independent stable subordinator $S_{\theta}(t)$ having characteristic function
\begin{equation}\label{subchar}\mathbb{E}\left[e^{i\gamma S_{\theta}(t)}\right]=e^{-t\lvert \gamma\lvert^{\theta}e^{-i\frac{\theta\pi}{2}\sgn{\gamma}}}.\end{equation}
\noindent The characteristic function of the subordinated pseudo-process $$Y_{\alpha,\theta}(t)=X_{\alpha}(S_{\theta}(t))$$ is given by
\begin{align}\label{subpseudochar}\mathbb{E}\left[e^{i\gamma Y_{\alpha,\theta}(t)}\right]=&\mathbb{E}\left[\mathbb{E}\left[e^{i\gamma X_{\alpha}(S_{\theta}(t))}\big\lvert S_{\theta}(t)\right]\right]\nonumber\\
=&\mathbb{E}\left[e^{-iS_{\theta}(t)\lvert \gamma\lvert^{\alpha}\sgn\gamma}\right]\nonumber\\
=&e^{-t\lvert \gamma\lvert^{\alpha\theta}e^{i\frac{\theta\pi}{2}\sgn\gamma}}\nonumber\\
=&e^{-t\lvert \gamma\lvert^{\alpha\theta}\cos\frac{\pi\theta}{2}\left(1+i\tan\frac{\pi\theta}{2}\sgn\gamma\right)}\end{align}
where we have used formulas (\ref{fractionalcharfun}) and (\ref{subchar}).\\
\noindent The characteristic function of a stable process $S_{\nu}(\sigma, \beta, \mu; t)$ of exponent $\nu\neq1$ reads
\begin{equation}\label{stableprocessescharfun}\mathbb{E}\left[e^{i\gamma S_{\nu}(\sigma, \beta, \mu; t)}\right]=e^{-\sigma^{\nu}\lvert \gamma\lvert^{\nu}\left(1-i\beta\sgn\gamma\tan\frac{\pi\nu}{2}\right)+i\mu\gamma}\end{equation}
\noindent where $0<\nu<2$ is the exponent of the stable process, $\sigma>0$ is the dispersion parameter, $\beta\in[-1,1]$ is the skewness parameter and $\mu\in\mathbb{R}$ is the location parameter. By comparing formulas (\ref{subpseudochar}) and (\ref{stableprocessescharfun}) we obtain that
\begin{equation}\label{indistribution}Y_{\alpha,\theta}(t)\overset{i.d.}{=}S_{\nu}(\sigma,\beta,\mu;t)\end{equation}
\noindent with
\begin{equation}\label{paramvalues}\nu=\alpha\theta,\qquad\beta=-\frac{\tan\frac{\pi\theta}{2}}{\tan\frac{\pi\alpha\theta}{2}},\qquad\sigma=\left(\cos\frac{\pi\theta}{2}\right)^{\frac{1}{\nu}},\qquad\mu=0.\end{equation}
\noindent In order for the relationship (\ref{indistribution}) to hold, we must have that $\nu\in(0,2]$ and $\beta\in[-1,1]$. This poses no problem since, given $\nu\in(0,2)$ and $\beta\in(0,1)$, it is always possible to choose $\theta\in(0,1)$ and $\alpha>1$ such that the formulas in (\ref{paramvalues}) are satisfied. As we will see, the extension to negative values of $\beta$ is straightforward. This permits us to represent stable processes as pseudo-processes time-changed with stable subordinators. For stable processes with exponent $\nu$, $1<\nu<2$ and skewness parameter $\beta\in(0,1)$, we obtain the following formula from the proof of theorem \ref{stablethm}:
\begin{equation}\label{PY}\mathbb{P}(Y_{\alpha,\theta}(t)\in dx)/dx=-\frac{1}{\pi x}\sum_{k=1}^{\infty}\frac{1}{k!}\sin\left(k\pi\frac{\alpha-1}{2\alpha}\right)\Gamma\left(1+\frac{k}{\theta\alpha}\right)\left(-\frac{x}{t^{\frac{1}{\theta\alpha}}}\right)^k.\end{equation}
\noindent Our representation (\ref{PY}) of the probability density function of an asymmetric stable process is consistent with that reported by Zolotarev \cite{zolotarev} (see theorem 2.4.2).\\
\noindent For $\nu=\alpha\theta>1$, the pseudo-distribution (\ref{PY}) represents the fundamental solution to the fractional partial differential equation
\begin{equation}\label{RFeqFin}
\begin{dcases}
\frac{\partial u}{\partial t}(x,t)={}_{x}D^{\alpha\theta}_{\theta}u(x,t),\qquad \alpha\theta>1,\;\theta\in(0,1)\\
u(x,0)=\delta(x).
\end{dcases}\end{equation}
\noindent This can be easily proved by showing that the Fourier transform of the solution to equation (\ref{RFeqFin}) coincides with the characteristic function (\ref{subpseudochar}). However, as formulas (\ref{paramvalues}) show, the subordinated pseudo-process $Y_{\alpha,\theta}(t)$ is identical in distribution to a genuine stochastic process only for suitable choices of $\alpha$ and $\theta$. The Cauchy process can be obtained as a particular case by setting $\theta=\frac{1}{\alpha}$, which yields
\begin{align}\mathbb{E}\label{cauchycf}\left[e^{i\gamma Y_{\alpha,\theta}(t)}\right]=&e^{-t\lvert \gamma\lvert \cos\frac{\pi}{2\alpha}\left(1+i\tan\frac{\pi}{2\alpha}\sgn\gamma\right)}\nonumber\\
=&e^{-i\gamma t\sin\frac{\pi}{2\alpha}-t\lvert \gamma\lvert \cos\frac{\pi}{2\alpha}}.\end{align}
\noindent Formula (\ref{cauchycf}) represents the characteristic function of a Cauchy process with probability density function
\begin{equation}\label{cauchydensity}\mathbb{P}(Y_{\alpha,\theta}(t)\in dx)/dx=\frac{1}{\pi}\;\frac{t\cos\frac{\pi}{2\alpha}}{t^2+2xt\sin\frac{\pi}{2\alpha}+x^2}.\end{equation}

\noindent The maximum of the Cauchy density (\ref{cauchydensity}) is located on the negative half-axis as shown in figure \ref{fig:cauchy}.
\begin{figure}[t]
\centering
\includegraphics[scale=0.9]{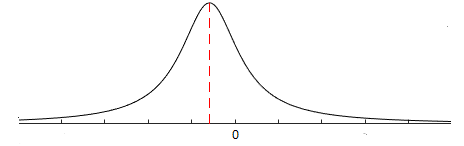}
\vspace{-3mm}
\caption{the Cauchy distribution obtained as a subordinated pseudo-process has its modal value on the negative half-axis.}
\label{fig:cauchy}
\end{figure}

\noindent We observe that the skewness parameter in (\ref{paramvalues})
$$\beta=-\frac{\tan\frac{\pi\theta}{2}}{\tan\frac{\pi\alpha\theta}{2}}$$
is positive for $0<\theta<1$ and $1<\alpha\theta<2$. Formula (\ref{PY}) thus describes the probability density function of a stable process with positive skewness parameter. By using the well-known property of stable processes
$$-S_{\nu}(\sigma,\beta,\mu;t)\overset{i.d.}{=}S_{\nu}(\sigma,-\beta,-\mu;t),\qquad \nu\neq 1$$
\noindent the series representation (\ref{PY}) can be immediately extended to stable processes with negative skewness parameter.

\end{document}